\documentclass{elsarticle}

\usepackage{lineno,hyperref}
\modulolinenumbers[5]

\journal{Statistics \& Probability Letters}

\biboptions{authoryear}

\makeatletter
\def\ps@pprintTitle{%
   \let\@oddhead\@empty
   \let\@evenhead\@empty
   \def\@oddfoot{\reset@font\hfil\thepage\hfil}
   \let\@evenfoot\@oddfoot
}
\makeatother

\usepackage{amsmath,amssymb,amsthm,mathtools}
\allowdisplaybreaks

\newtheorem{theorem}{Theorem}

\newtheorem{proposition}[theorem]{Proposition}
\theoremstyle{definition}
\newtheorem{definition}[theorem]{Definition}
\theoremstyle{remark}

\newcommand{\R}{\mathbb{R}}
\renewcommand{\d}{}
\newcommand{\ind}{\mathbf 1}
\newcommand*{\norm}[1]{\left\lVert#1\right\rVert}
\DeclareMathOperator{\supp}{supp}

\begin{document}

\begin{frontmatter}

\title{On mild and weak solutions for stochastic heat equations with piecewise-constant conductivity}

\author[a]{Yuliya Mishura}
\ead{myus@univ.kiev.ua}

\author[a]{Kostiantyn Ralchenko}
\ead{k.ralchenko@gmail.com}

\author[b]{Mounir Zili}
\ead{Mounir.Zili@fsm.rnu.tn}

\address[a]{Taras Shevchenko National University of Kyiv, Department of Probability Theory, Statistics and Actuarial Mathematics, Volodymyrska 64/13, 01601 Kyiv, Ukraine}
\address[b]{University of Monastir, Faculty of sciences of Monastir,
Department of Mathematics, Avenue de l'environnement, 5019 Monastir,
Tunisia}

\begin{abstract}
We investigate a stochastic partial differential equation with second order elliptic operator in divergence form, having a piecewise constant
diffusion coefficient, and driven by a space--time white noise.
We introduce a notion of weak solution of this equation and prove its equivalence to the already known notion of mild solution.
\end{abstract}


\end{frontmatter}


\section{Introduction}

Since the pioneering work of \citet{W}, investigation of solutions of stochastic partial differential equations (SPDE)  has raised the interest of many researchers, especially for their numerous applications   (see \citet{Dal,Khoshnevisan} and references therein).   In fact, there are two main classes of  such solutions; classical or generalized ones. A classical solution is a  function, sufficiently smooth,   satisfying the  equation and its initial condition point-wise on the same set of probability one. Every solution that is not classical is usually called generalized.
A generalized solution extends certain properties of a classical solution without requiring existence
of partial derivatives. That is why, there exist various notions of generalized solutions, including mild and weak ones. The comparison between such notions has been dealt with only in very few cases, such as  in the case of the standard stochastic heat equation \citep[see, e.\,g.,][]{Khoshnevisan}. This can be explained by the fact that,  in general, there is no obvious reason to claim that the mild and weak solutions define the same object.
In this paper we address this question in the case of the  following equation:
\begin{equation}\label{e1:1}
\left\{
\begin{aligned}
\frac{\partial u(t, x)}{\partial t} &= {\mathcal L} u(t, x) +\dot{W} (t,x) ; \; \; t \in ( 0, T], \; x\in \R ,\\
u(0,\cdot) &:= 0,\;\;\;x\in\R.
\end{aligned}\right.
\end{equation}
 Here
$\dot{W}$  denotes a  ``derivative'' of a centered Gaussian field
$  {W} = \{  {W}(t,C); t \in [0,T ], C \in {\mathcal B}_b({\mathbb R}) \}$ with covariance
\begin{equation}
\label{e:2} {\mathbb E} ( {W}(t,C)  {W}(s,D)) = (t \wedge s) \lambda (C \cap D),
\end{equation}
where $\lambda$ is the Lebesgue measure on $\mathbb R$, and ${\mathcal L}$ is the operator defined by
\[
 {\mathcal L} = \frac{1}{2 \rho (x)} \frac{d}{dx} \left( \rho (x) A(x) \frac{d}{dx} \right) ,
\]
\begin{equation}
\label{eq:coefA}
A(x)=  a_1 {\mathbf 1}_{\{ x \le 0 \} } + a_2 {\mathbf 1}_{\{ 0 < x  \} }  \quad \hbox{and} \quad
\rho(x)= \rho_1 {\mathbf 1}_{\{ x \le 0 \} } + \rho_2 {\mathbf 1}_{\{ 0 < x  \} },
\end{equation}
$a_i, \rho_i$ ($i=1, 2$) are  strictly positive constants,
and
 $\displaystyle \frac{df}{dx}$ denotes the derivative  of $f$ in the distributional sense.
Equation  (\ref{e1:1}) represents a natural extension of the stochastic heat equation driven by space-time white noise, which has been widely studied in the literature (see, e.\,g., \citet{Dal,Khoshnevisan,W} and  references therein).
Equation (\ref{e1:1}) has been introduced in \citet{ZZ} because of its interest  in modeling  diffusion  phenomena in medium consisting of two kinds of materials,  undergoing stochastic perturbations. In  \citet{ZZ}, the authors  proved  the existence of the mild solution to (\ref{e1:1}), they presented explicit expressions of its  covariance and variance functions, and they analyzed some regularity properties of its sample paths.
Then, \citet{ZZ2} presented an estimation method   of the parameters  $a_1$ and $a_2$ appearing in (\ref{eq:coefA}) and, after that, in \citet{Zili-Zougar}, they made a deep study of the spatial quadratic variations of the mild solution process.

We make here an interesting new step in the study of SPDE (\ref{e1:1}), by introducing a notion of weak solution of this equation and then, by showing its equivalence with the already known notion of mild solution. Our proofs  require a stochastic Fubini theorem version and  some non-random  partial differential equations characteristics;   they are particularly  based on the use of the known expression of the fundamental solution related to the operator ${\mathcal L}$ and some of its characteristics.

The paper is organized as follows.
 In the first part of the next section  we introduce  and explain the notion of weak solution to SPDE (\ref{e1:1}) that we will deal within this paper. Then,  the rest of the paper is devoted to the proof of the equivalence between  the notions of mild and weak solutions for Equation (\ref{e1:1}).

\section{Weak and mild solutions}

\subsection{Weak Solution}
Let us suppose formally for the moment that equation  (\ref{e1:1}) admits  a classical solution $u\colon\Omega\times[0, \infty ) \times {\mathbb R}\to\R$ that is  a function belonging to $ C^2([0, \infty ) \times {\mathbb R})$ and satisfying equation $(\ref{e1:1})$ on a set $\Omega'\in\Omega$ of probability $1$.

Then, for every $ \varphi \in C^{\infty}_c([0,T]\times \R )$ and for every $\omega\in\Omega'$
\begin{multline*}
\underbrace{\int_{[0,T] \times \R}\frac{\partial u(s, x)}{\partial s}\;\varphi(s,x)\; \rho (x)\,dx\,ds}_{I}-\underbrace{\int_{[0,T]\times \R}{\mathcal L} u(s, x)\;\varphi(s,x)\; \rho (x)\,dx\,ds}_{J}
\\
=\int_{[0,T] \times \R} \varphi(s,x)\, \rho (x)\;W(ds,dx).
\end{multline*}
An integration by parts allows us to obtain:
\begin{align*}
 I &=     \int_{\R}\Big[ u(s,x)\,\varphi(s,x)\Big]^{T}_{0}\, \rho (x)\,dx-\int_{[0,T]\times \R}u(s,x)\;\frac{\partial \varphi(s, x)}{\partial s}\; \rho (x)\,dx\,ds \\
&=   -\int_{[0,T] \times \R}u(s,x)\;\frac{\partial \varphi(s, x)}{\partial s}\; \rho (x)\,dx\,ds.
\end{align*}
As for the integral $J$ we have:
\begin{align*}
 J &=  \frac{1}{2} \int_{[0,T]\times \R}\frac{\partial}{\partial x} \left( \rho (x) A(x)  \frac{\partial u(s,x)}{\partial x} \right)\;\varphi(s,x)\,ds\,dx \\
&= - \frac{1}{2} \int_{[0,T]\times \R} \rho(x) A(x) \frac{\partial u(s,x)}{\partial x} \; \frac{\partial}{\partial x} \left( \varphi(s,x) \right) \,ds\,dx
\end{align*}
where in the last equality we  used the definition of the derivative in the distribution sense of the locally integrable function
$  x \mapsto  \rho (x) A(x)  \frac{\partial u(s,x)}{\partial x} $.

Consequently,
\begin{align*}
J &=    - \frac{1}{2} \int_{[0,T]\times \R^\star } \rho(x) A(x) \frac{\partial u(s,x)}{\partial x} \; \frac{\partial}{\partial x} \left( \varphi(s,x) \right) \,ds\,dx \\
&=  - \frac{1}{2} \int_{[0,T]\times (0, + \infty )} \rho_2 a_2 \frac{\partial u(s,x)}{\partial x} \; \frac{\partial}{\partial x} \left( \varphi(s,x) \right) \,ds\,dx \\
&\quad   - \frac{1}{2} \int_{[0,T]\times (-\infty , 0)} \rho_1 a_1 \frac{\partial u(s,x)}{\partial x} \; \frac{\partial}{\partial x} \left( \varphi(s,x) \right) \,ds\,dx
\end{align*}

By an integration by parts we get:
\[
\int_{(0, +\infty)} \frac{\partial u(s,x)}{\partial x} \; \frac{\partial \varphi(s,x)}{\partial x}  \,dx = - u(s,0)\frac{\partial \varphi(s,0)}{\partial x} - \int_{(0, +\infty)}  u(s,x) \; \frac{\partial^2 \varphi(s,x)}{\partial ^2x}  \,dx,
\]
and
\[
\int_{(-\infty , 0)} \frac{\partial u(s,x)}{\partial x} \; \frac{\partial \varphi(s,x)}{\partial x}  \,dx =  u(s,0)\frac{\partial \varphi(s,0)}{\partial x} - \int_{(-\infty , 0)}  u(s,x) \; \frac{\partial^2 \varphi(s,x)}{\partial ^2x}  \,dx.
\]

Thus,
\begin{align*}
 J &=  \frac{1}{2} \int_{[0,T]} (\rho_2a_2 -\rho_1 a_1)u(s,0)\frac{\partial \varphi(s,0)}{\partial x} ds
   + \frac{1}{2} \int_{[0,T]\times \R^\star }  u(s,x) \rho(x) A(x) \frac{\partial^2 \varphi(s,x)}{\partial ^2x}  \,dxds \\
&=  \frac{1}{2} \int_{[0,T]} (\rho_2a_2 -\rho_1 a_1)u(s,0)\frac{\partial \varphi(s,0)}{\partial x} ds
+  \int_{[0,T]\times \R^\star }  u(s,x) {\mathcal L} ( \varphi(s,x)) \rho(x) \,dxds
\end{align*}

Therefore,
\begin{multline}
\label{weaksol}
 -\int_{[0,T] \times \R ^\star }u(s,x)\,\Bigg[ \frac{\partial \varphi(s, x)}{\partial s}+ \mathcal{L}\varphi(s,x)\Bigg]\,\rho (x)\,dx\,ds =  \int_{[0,T] \times \R} \varphi(s,x) \rho (x)\,W(ds,dx) \\
  + \frac{1}{2} \int_{0}^T (\rho (0^+) A(0^+) -\rho (0^-) A(0^-))u(s,0)\frac{\partial \varphi(s,0)}{\partial x} ds,
\end{multline}

This leads to the first definition   of the generalized solution    to  the  SPDE $(\ref{e1:1})$. It is called a weak solution.

\begin{definition}
A stochastic process $ u:= \{ u(t,x)\}_{\{t\ge 0,x\in \R \}}$ is a weak solution to the stochastic partial differential equation $(\ref{e1:1})$ if $ u\in L_{Loc}^1(\R_+\times\R, dx\,dt)$,   and  for all $ T > 0$ and   every $\varphi \in C^{\infty}_c([0,T]\times \R ),$
$u$ satisfies equation (\ref{weaksol}).
\end{definition}

In the following proposition we present the expression of the fundamental solution associated to the SPDE (\ref{e:2}). For a proof see, e.g., \citet{Zi, Zil} and \citet{CZ}.

\begin{proposition}
The fundamental solution $G$ of the partial differential equation (\ref{e1d})
 is given by
\begin{multline} \label{e:6}
G( t, x,  y)
 = \left [ \frac{1}{\sqrt{2 \pi t}}
 \left( \frac{ {\mathbf 1}_{ \{ y \le 0 \} } }{\sqrt{a_1}}
+ \frac{ {\mathbf 1}_{ \{ y > 0 \} } }{\sqrt{a_2}} \right) \times \left\{ \exp \left( - \frac{(f(x) - f(y))^2}
{2 t}\right) \right.\right. \\
 +  \left.\left.
\frac{\sqrt{a_1}+ \sqrt{a_2} (\alpha -1)}{\sqrt{a_1}-\sqrt{a_2} (\alpha -1)}\, {\rm sign} (y)\, \exp
\left( - \frac{(\mid f(x) \mid  + \mid f(y) \mid)^2}{2t}\right) \right\} \right ] {\mathbf 1}_{0 < t},
\end{multline}
where
\[
f(y)= \frac{y}{\sqrt{a_1}}{\mathbf 1}_{\{y\leq 0\}}+\frac{y}{\sqrt{a_2}} {\mathbf 1}_{\{y>0\}} \; \;
\text{and} \;  \; \alpha  = 1- \frac{\rho_1 a_1}{\rho_2a_2}.
\]
\end{proposition}

Now we introduce  the second definition of the generalized solution    to the  SPDE  $(\ref{e1:1})$ that is called a mild solution.

\begin{definition}
\label{defMild}
A stochastic process $u:= \{ u(t,x)\}_{\{ t \ge 0, x\in \R \}}$ is a mild solution to the stochastic partial differential equation  $(\ref{e1:1})$  if it can be written in the following integral form:
\[
 u(t,x) := \int_{0}^{t}\!\!\int_{\R}G(t-s,x, y)\,W(ds,\,dy),
\]
where $G$ denotes the fundamental solution to the  non random PDE
\begin{equation}\label{e1d}
 \frac{\partial u(t, x)}{\partial t} :=   {\mathcal L} u(t, x) .
\end{equation}
\end{definition}

The purpose of this paper is to investigate the equivalence between mild and weak solutions.

\subsection{Mild solution implies weak solution}

In this subsection we  prove that every mild solution is also a weak solution.

%

\begin{theorem}\label{th:m->w}
Every mild solution to equation (\ref{e1:1}) is also  its  weak solution.
\end{theorem}

\begin{proof}
Let $u$ be a mild solution and $\varphi \in C^{\infty}_c([0,T]\times \R)$.
 Introduce the operator
 \[
 H : =  -\frac{\partial }{\partial s} - \mathcal{L} .
 \]

With this notation, and for any fixed $T>0$ we have that
\begin{multline}
\label{eqWeak1}
\int_{[0,T]\times \R ^\star}u(s,x)\,H(\varphi) (s,x)\,\rho (x)\,ds\,dx
\\
= \int_{[0,T]\times \R ^\star}   \int_{[0,T]\times \R} G(s-u,x,y)\,W(du, \,dy)  H(\varphi) (s,x) \,\rho (x)\,ds\,dx.
\end{multline}

Now we apply the stochastic Fubini theorem for worthy martingale measures \citep[Th.~2.6]{W} \citep[see also][Th.~5.30]{Dal}.
 According to \citet[p.~20]{Dal}, the white noise $W$ is a worthy martingale measure with dominating measure $K$ defined by $K(dx,\,  dy, \, ds):= dx\, dy \, ds.$ In this case, we can apply the stochastic Fubini theorem for worthy martingale measures to the functions integrable with respect to the product of  measures. 
The function that will be integrated, equals
\[
g( s, u, x,y)=G(s-u,x,y)  \, H(\varphi)(s,x)\,\rho(x) {\mathbf 1}_{\R \setminus \{ 0 \}} (x),
\]
 $  (s, u, x,y) \in   [0,T]^2\times \R^2$.
 This function is measurable, integrable with respect to the product of  measures, and therefore
\begin{align*}
\MoveEqLeft \int_{  [0,T]^2\times \R^3}
\left|g( s, u, x,y)\,g( s, u, z,y)    \right|\,K(ds,\, dx,\,dz)\, du\,dy
\\
&=  \int_{ [0,T]^2\times \R^3} \left| G(s-u,x,y)  \, H(\varphi)(s,x)\,\rho(x)  {\mathbf 1}_{\R \setminus \{ 0 \} }(x) \,G(s-u,z,y)  \right.
\\
&\quad  \times \left. H(\varphi)(s,z)\,\rho(z) {\mathbf 1}_{\R \setminus \{ 0 \} }(z) \right|ds\,du\,dx\, dy\, dz
\\
&\leq  [\max(\rho_1,\rho_2 ) ]^2 \int_{[0,T]^2\times \R^3}  \left| G(s-u,x,y)  \,H(\varphi)(s,x)) {\mathbf 1}_{\R \setminus \{ 0 \} }(x) \,G(s-u,z,y) \right.
\\
&\quad  \times \left.H(\varphi)(s,z) {\mathbf 1}_{\R \setminus \{ 0 \} }(z)\right| ds\,du\,dx\, dz\, dy.
\end{align*}
Moreover,  on the one hand, from Expression (\ref{e:6}) of the fundamental solution $G$ and the fact that the terms $  \exp \left( - \frac{(f(x) - f(y))^2}{2 t}\right)$ and $  \exp \left( - \frac{(\mid f(x) \mid  + \mid f(y) \mid)^2}{2t}\right)$ are bounded by $1$,
we have $ \left| G(s-u,z,y) \right| \le \frac{C}{\sqrt{s-u}} {\mathbf 1}_{u < s}$ for every $z,  y \in {\R}$, where $C $ is a constant depending only on the coefficient $\beta , a_1$ and $a_2$. On the other hand, by Corollary 2.1 from \citet{Mish-Kosta-Zili-Zougar}, we have that
\begin{equation}\label{eq:bound-int-G}
\int_{\R } \left| G(s-u,x,y) \right|dy \le C,
\end{equation}
for every $ s,u  \in [0,T]$ and $x, y \in \R.$

All this implies the inequalities
\begin{align*}
F&:=\int_{ [0,T]^2\times \R^3 }
\left|g(s,u,x,y)\,g(s,u,z,y)    \right|\,K(ds,\, dx,\,dz)\,  du dy
\\
&\leq  C \int_{ [0,T]^2\times \R^2} \frac{1}{\sqrt{s-u}}  {\mathbf 1}_{u < s} \left|\,H(\varphi)(s,x)) {\mathbf 1}_{\R \setminus \{ 0 \} }(x)\,  \,H(\varphi)(s,z) {\mathbf 1}_{\R \setminus \{ 0 \} }(z) \right|
\\
&\quad   \times \left.   \Big( \int_{\R } \left| G(s-u,x,y) \right|dy \Big) ds\, du\,dx\, dz  \right.
\\
&\leq  C \int_{ [0,T]^2\times \R^2} \frac{1}{\sqrt{s-u}}  {\mathbf 1}_{u < s} \left|\,H(\varphi)(s,x)) {\mathbf 1}_{\R \setminus \{ 0 \} }(x)\,  \,H(\varphi)(s,z){\mathbf 1}_{\R \setminus \{ 0 \} }(z) \right| ds\, du\, dx\, dz.
\end{align*}

Now denoting by $K_2$ a compact set in $\R$ such that $[0,T]\times K_2$ contains the compact support of $\varphi$,   we have:
$$\mid H(\varphi) (s,x)  {\mathbf 1}_{\R \setminus \{ 0 \} }(x) \mid \le C , $$
for every $(s,x) \in [0,T] \times K_2$.
Therefore,
\begin{align*}
 F&\leq   C \int_{[0,T]^2 \times K_2^2} \frac{1}{\sqrt{s-u}}  {\mathbf 1}_{u < s} \left|\,H(\varphi)(s,x)) {\mathbf 1}_{\R \setminus \{ 0 \} }(x)\,  \,H(\varphi)(s,z) {\mathbf 1}_{\R \setminus \{ 0 \} }(z) \right| ds\, du\, dx\, dz
\\
&\leq C \int_{[0,T]^2 } \frac{1}{\sqrt{s-u}}  {\mathbf 1}_{u < s} \, ds\, du  <   \infty .
\end{align*}
Thus all the hypothesis of the Fubini theorem \citep[Th.~2.6]{W} are satisfied and consequently, from (\ref{eqWeak1}) we get
\begin{multline*}
\int_{[0,T]\times \R^\star }u(s,x)\,H(\varphi) (s,x)\,\rho (x)\,dx\,ds \\
=\int_{[0,T)\times \R}  \Bigg[  \int_{[0,T]\times \R^\star  } G(s-u,x,y)\,  H(\varphi) (s,x) \,\rho (x)\,ds\,dx  \Bigg] W(du ,\,dy) .
\end{multline*}

We have
\begin{multline*}
 \int_{[0,T] \times \R^\star } G(s-u,x,y)  \,\mathcal{H}(\varphi)(s,x)\,\rho (x) \,ds\,dx
=\int_{[0,T-u] \times \R^\star } G(s,x,y)  \,\mathcal{H}(\varphi)(s+u,x)\,\rho (x)\,ds\,dx
\\
= \int_{[0,\epsilon)\times\R^\star} G(s,x,y)  \,\mathcal{H}(\varphi)(u+s,x)\,\rho (x)\,ds\,dx
+ \int_{[\epsilon,T-u]\times\R^\star } G(s,x,y)  \,\mathcal{H}(\varphi)(u+s,x)\,\rho (x)\,ds\,dx
\end{multline*}
for every $0 < \epsilon < T -u$.

Let us denote
\[
J_{1,\epsilon}=\int_{[0,\epsilon)\times\R^\star} G(s,x,y)  \,\mathcal{H}(\varphi)(u+s,x)\,\rho (x)\,ds\,dx
\]
and
\[
J_{2,\epsilon}=\int_{[\epsilon,T-u]\times\R^\star } G(s,x,y)\,\mathcal{H}\varphi(u+s,x)\,\rho (x)\,ds\,dx.
\]
As for  $J_{1,\epsilon}$ we have:
\begin{align*}
 J_{1,\epsilon} &=  \int_{(0,\epsilon)\times\R^\star } G(s,x,y)  \,\mathcal{H}(\varphi)(u+s,x)\,\rho (x)\,ds\,dx \\
 &=     - \int_{(0,\epsilon)\times [0, + \infty )} G(s,x,y)  \, \Big( \frac{\partial }{\partial s}\varphi(u+s,x) \rho(x)  + \rho_2 a_2 \frac{\partial ^2 }{\partial x^2}  \varphi(u+s,x) \Big)\,\,ds\,dx  \\
&\quad  - \int_{(0,\epsilon)\times (- \infty , 0 )} G(s,x,y) \Big( \frac{\partial }{\partial s}\varphi(u+s,x) \rho(x)  + \rho_1 a_1 \frac{\partial ^2 }{\partial x^2}  \varphi(u+s,x) \Big)\, ds\,dx.
 \end{align*}
 Thus,
\[
 \vert J_{1,\epsilon} \vert
\le      C  \int_{(0,\epsilon)} \int_{\R}\Big \vert G(s,x,y)\Big\vert \,ds\,dx
\]
with
\[
C =  \max (\rho_1, \rho_2) \sup_{(t,x) \in \R^+ \times \R}\left\vert  \frac{\partial \varphi(\cdot,x)}{\partial t}\right\vert+ \max(a_1\rho_1, a_2\rho_2 ) \sup_{ (t,x) \in \R^+ \times \R } \left\vert \frac{\partial^2\varphi(t,\cdot)}{\partial x^2} \right\vert.
\]
 This with \eqref{eq:bound-int-G} leads to
\[
 \vert  J_{1,\epsilon} \vert \le   C  \epsilon ,
\]
and consequently, $ \lim_{\epsilon \rightarrow 0} J_{1,\epsilon} = 0$.

Concerning  $J_{2,\epsilon}$ we have:
\begin{align*}
 J_{2,\epsilon} &=  \int_{[\epsilon,T-u]\times\R^\star } G(s,x,y)\,\mathcal{H} (\varphi)(u+s,x)\,\rho(x)\,ds\,dx \\
&=   -\int_{[\epsilon,T-u]\times\R^\star} G(s,x,y)\;\frac{\partial \varphi(u+s,x)}{\partial s}\,\rho(x)\,ds\,dx \\
&\quad - \int_{[\epsilon,T-u]\times\R^\star } G(s,x,y)\,\mathcal{L}\varphi(u+s,x)\,\rho(x)\,ds\,dx\\
&= A_{\epsilon}+ B_{\epsilon}.
\end{align*}

Integrating by parts, we get:
\begin{align*}
A_{\epsilon}&= \int_{\R^\star }\Big[- \varphi(u+s,x)\,G(s,x,y)\,\rho(x) \Big]^{T-u}_{\epsilon}\,dx+\int_{[\epsilon,T-u]\times\R^\star } \varphi(u+s,x)\, \frac{\partial G(s,x,y)}{\partial s}\,\rho(x)ds\,dx
\\
&=  \int_{\R^\star }\varphi(\epsilon+u,x)\,G(\epsilon,x,y)\,\rho(x)\,dx -\int_{\R^\star }\varphi(T,x)\,G(T-u,x,y)\,\rho(x)\,dx
\\
&\quad +
\int_{[\epsilon,T-u]\times\R^\star } \varphi(u+s,x)\, \frac{\partial G(s,x,y)}{\partial s}\,\rho(x)\,ds\,dx
\\
&=\int_{\R^\star }\varphi(\epsilon+u,x)\,G(\epsilon,x,y)\,\rho(x)\,dx + \int_{[\epsilon,T-u]\times\R^\star } \varphi(u+s,x)\, \frac{\partial G(s,x,y)}{\partial s}\,\rho(x)\,ds\,dx,
\end{align*}
since $\varphi( T,\cdot) = 0$.

For $B_\epsilon$ we have,
\begin{multline*}
\int_{\R^\star} \frac{\partial}{\partial x} \left( \rho (x)A(x) \frac{\partial \varphi(s,x)}{\partial x} \right)\; G(s-u,x,y)\,dx
\\
=   \int_{(0, +\infty )} \rho_2 a_2 \frac{\partial^2 \varphi(s,x)}{\partial x^2}  G(s-u,x,y)\,dx  +  \int_{(- \infty , 0)} \rho_1 a_1 \frac{\partial^2 \varphi(s,x)}{\partial x^2}  G(s-u,x,y)\,dx.
\end{multline*}

Therefore
\begin{align*}
B_{\epsilon}&=  - \frac{1}{2} \int_{[\epsilon+u,T]\times (0, + \infty )} \rho_2a_2 \frac{\partial^2 \varphi(s,x)}{\partial x^2} \;  G(s-u,x,y)  \,ds\,dx
\\
&\quad  - \frac{1}{2} \int_{[\epsilon+u,T]\times (-\infty , 0)} \rho_1 a_1 \frac{\partial^2 \varphi(s,x)}{\partial x^2} \;   G(s-u,x,y)  \,ds\,dx.
\end{align*}

Integrating by parts we get
\[
\int_{-\infty}^{0}  \frac{\partial^2 \varphi(s,x)}{\partial x^2} \;  G(s-u,x,y)  dx = G(s-u,0,y) \frac{\partial \varphi(s,0)}{\partial x} - \int_{-\infty}^0  \frac{\partial \varphi(s,x)}{\partial x} \; \frac{\partial G(s-u,x,y) }{\partial x}  dx
\]
and
\[
\int_0^{+\infty}  \frac{\partial^2 \varphi(s,x)}{\partial x^2} \;  G(s-u,x,y)  dx = - G(s-u,0,y) \frac{\partial \varphi(s,0)}{\partial x} - \int_0^{+\infty}  \frac{\partial \varphi(s,x)}{\partial x} \; \frac{\partial G(s-u,x,y) }{\partial x}  dx
\]

Thus,
\begin{multline*}
 B_{\epsilon} =    \frac{1}{2} \int_{\epsilon + u}^T (\rho_2a_2 - \rho_1a_1)
G(s-u,0,y) \frac{\partial \varphi(s,0)}{\partial x} ds
\\
+  \frac{1}{2} \int_{[\epsilon+u,T]\times\R ^\star} \rho(x) A(x) \frac{\partial \varphi(s,x)}{\partial x} \; \frac{\partial G(s-u,x,y) }{\partial x}  dx ds.
\end{multline*}

On another side,  we have,
\begin{align*}
\MoveEqLeft   \int_{\R^\star }  \varphi(s,x) \; {\mathcal L} G(s-u,x,y) \rho (x)   dx
=
\frac{1}{2} \int_{\R^\star }  \varphi(s,x) \; \frac {\partial }{\partial x}  \left( A(x) \rho (x) \frac{\partial G(s-u,x,y)}{\partial x} \right)  dx
\\
&=
\frac{1}{2} \int_{(0, +\infty )}  \varphi(s,x) \; a_2 \rho_2 \frac {\partial^2 G(s-u,x,y)}{\partial x^2} dx
+ \frac{1}{2} \int_{(-\infty , 0)}  \varphi(s,x) \; a_1 \rho_1 \frac {\partial^2 G(s-u,x,y)}{\partial x^2}  dx
\\
&= - \frac{1}{2} \int_{\R^\star }  \frac{\partial \varphi(s,x)}{\partial x} \;  A(x) \rho (x) \frac{\partial G(s-u,x,y)}{\partial x}  dx
\\
&\quad + \frac{1}{2} \left(a_1\rho_1 \frac{\partial G(s-u,0^-,y)}{\partial x}- a_2\rho_2 \frac{\partial G(s-u,0^+,y)}{\partial x} \right) \varphi(s,0)
\end{align*}
In other words,
\begin{multline*}
\frac{1}{2} \int_{\R^\star }  \frac{\partial \varphi(s,x)}{\partial x} \;  A(x) \rho (x) \frac{\partial G(s-u,x,y)}{\partial x}  dx  =
 -  \int_{\R^\star }  \varphi(s,x) \; {\mathcal L} G(s-u,x,y) \rho (x)   dx  \\
+  \frac{1}{2} \left(a_1\rho_1 \frac{\partial G(s-u,0^-,y)}{\partial x}- a_2\rho_2 \frac{\partial G(s-u,0^+,y)}{\partial x} \right) \varphi(s,0)
\end{multline*}

Therefore
\begin{align*}
 B_\epsilon &=   \frac{1}{2} \int_{\epsilon + u}^T (\rho_2a_2 - \rho_1a_1)
G(s-u,0,y) \frac{\partial \varphi(s,0)}{\partial x} ds
\\
&\quad -\int_{[\epsilon+u,T]\times\R^\star }\varphi(s,x) \; {\mathcal L} G(s-u,x,y) \rho (x)   dx ds
\\
&\quad  + \int_{\epsilon + u}^T
   \frac{1}{2} \left(a_1\rho_1 \frac{\partial G(s-u,0^-,y)}{\partial x}- a_2\rho_2 \frac{\partial G(s-u,0^+,y)}{\partial x} \right) \varphi(s,0)
ds.
\end{align*}
and consequently,
\begin{align*}
 J_{2,\epsilon}&=  A_{\epsilon}+ B_{\epsilon}
 \\
&=\int_{\R}\varphi(\epsilon+u,x)\,G(\epsilon,x,y)\,\rho(x)\,dx  \\
&\quad  + \int_{(\epsilon,T-u)\times\R ^\star} \varphi(u+s,x)\, \left[\frac{\partial G(s,x,y)}{\partial s}-\mathcal{L}G(s,x,y)\right]\,\rho(x)dx\,ds\\
&\quad   + \frac{1}{2} \int_{\epsilon + u}^T (\rho(0^+)A(0^+) - \rho(0^-)A(0^-))
G(s-u,0,y) \frac{\partial \varphi(s,0)}{\partial x} ds \\
&\quad   + \int_{\epsilon + u}^T
   \frac{1}{2} \left(A(0^-)\rho (0^-) \frac{\partial G(s-u,0^-,y)}{\partial x}- A(0^+)\rho (0^+) \frac{\partial G(s-u,0^+,y)}{\partial x} \right) \varphi(s,0)
ds.
\end{align*}

Let us calculate the term $ a_1 \rho_1 \frac{\partial G( t, 0^-,  y)}{\partial x} - a_2 \rho_2 \frac{\partial G( t, 0^+,  y)}{\partial x}$.

If $x > 0$, $y > 0$ and $t > 0$,
\[
  G( t, x,  y) =   \frac{1}{\sqrt{2 \pi t}\sqrt{a_2}}
 \times \left\{ \exp \left( - \frac{(x - y)^2}
{2 a_2 t}\right)   +   \beta  \exp
\left( - \frac{(x   + y)^2}{2a_2t}\right) \right\} ,
\]
where $\beta = \frac{\sqrt{a_1}+ \sqrt{a_2} (\alpha -1)}{\sqrt{a_1}-\sqrt{a_2}(\alpha -1)}$.

Thus,
\[
  \frac{\partial G( t, x,  y)}{\partial x} =   \frac{1}{\sqrt{2 \pi t}\sqrt{a_2}}
 \left\{ - \frac{x-y}{a_2t} \exp \left( - \frac{(x - y)^2}
{2 a_2 t}\right)   -   \beta  \frac{x+y}{a_2t} \exp
\left( - \frac{(x   + y)^2}{2a_2t}\right) \right\} .
\]
Therefore,
\begin{align*}
  \frac{\partial G( t, 0^+,  y)}{\partial x} &=   \frac{1}{\sqrt{2 \pi t}\sqrt{a_2}}
 \left\{ - \frac{-y}{a_2t} \exp \left( - \frac{y^2}
{2 a_2 t}\right)   -   \beta  \frac{y}{a_2t} \exp
\left( - \frac{y^2}{2a_2t}\right) \right\}
\\
&=    \frac{-y}{\sqrt{2 \pi t}a_2 t\sqrt{a_2}}
  (\beta -1)  \exp \left( - \frac{y^2}
{2 a_2 t}\right).
\end{align*}

If $x < 0$, $y > 0$ and $t > 0$, then
\begin{align*}
   G( t, x,  y) &=   \frac{1}{\sqrt{2 \pi t}\sqrt{a_2}}
 \left\{ \exp \left( - \frac{(\sqrt{a_2}x - \sqrt{a_1}y)^2}
{2 a_1a_2 t}\right)   +   \beta  \exp
\left( - \frac{(\sqrt{a_2} x   - \sqrt{a_1} y)^2}{2a_1a_2t}\right) \right\}
\notag\\
&=     \frac{1+ \beta }{\sqrt{2 \pi t}\sqrt{a_2}} \exp \left( - \frac{(\sqrt{a_2}x - \sqrt{a_1}y)^2}
{2 a_1a_2 t}\right) .
\end{align*}

Thus,
\[
  \frac{\partial G( t, x,  y)}{\partial x} =
  \frac{1+ \beta }{\sqrt{2 \pi t}\sqrt{a_2}} \frac{- \sqrt{a_2} (\sqrt{a_2}x - \sqrt{a_1}y)}{a_1a_2t} \exp \left( - \frac{(\sqrt{a_2}x - \sqrt{a_1}y)^2}
{2 a_1a_2 t}\right) .
\]

Therefore,
$$
  \frac{\partial G( t, 0^-,  y)}{\partial x} =
  \frac{1+ \beta }{\sqrt{2 \pi t}\sqrt{a_2}} \frac{\sqrt{a_2} \sqrt{a_1}y}{a_1a_2t} \exp \left( - \frac{y^2}
{2 a_2 t}\right) .
$$

All this implies that
\begin{align*}
\MoveEqLeft  a_1 \rho_1 \frac{\partial G( t, 0^-,  y)}{\partial x} - a_2 \rho_2 \frac{\partial G( t, 0^+,  y)}{\partial x}
\\
&=  a_1\rho_1\frac{1+ \beta }{\sqrt{2 \pi t}\sqrt{a_2}} \frac{\sqrt{a_2} \sqrt{a_1}y}{a_1a_2t} \exp \left( - \frac{y^2}
{2 a_2 t}\right) - a_2 \rho_2  \frac{-y}{\sqrt{2 \pi t}a_2 t\sqrt{a_2}}
  (\beta -1)  \exp \left( - \frac{y^2}
{2 a_2 t}\right)
\\
&=   \Bigg( \frac{ \rho_1 (1 + \beta ) \sqrt{a_1} }{\sqrt{a_2}}  + \rho_2(\beta -1) \Bigg)  \times  \frac{y}{\sqrt{a_2}t \sqrt{2\pi t}} \exp \left( - \frac{y^2}
{2 a_2 t}\right)
\\
&=  \frac{2\rho_1a_1 + 2 a_2 \rho_2 (\alpha -1)}{\sqrt{a_1}\sqrt{a_2} - a_2 (\alpha -1)}  \frac{y}{\sqrt{a_2}t \sqrt{2\pi t}} \exp \left( - \frac{y^2}
{2 a_2 t}\right).
 \end{align*}
 Using the fact that  $\alpha = 1 - \frac{\rho_1a_1}{\rho_2a_2},$ we easily see that
\[
\frac{2\rho_1a_1 + 2 a_2 \rho_2 (\alpha -1)}{\sqrt{a_1}\sqrt{a_2} - a_2 (\alpha -1)} =0.
\]
 Therefore,
\[
 a_1 \rho_1 \frac{\partial G( t, 0^-,  y)}{\partial x} - a_2 \rho_2 \frac{\partial G( t, 0^+,  y)}{\partial x} = 0,
\]
for every $y > 0$.
By similar calculation, we get the same result if  $y \le 0$.
Now, since $G$ is a fundamental solution of PDE (\ref{e1d}), we have  $\frac{\partial G(s,x,y)}{\partial s}-\mathcal{L}G(s,x,y) = 0,$ for every $s \in (0,T)$ and $x \in \R \setminus \{ y \}$ and consequently,
\begin{align*}
J_{2,\epsilon}&=\int_{\R}\varphi(\epsilon+u,x)\,G(\epsilon,x,y)\,\rho(x)\,dx \\
&\quad   + \frac{1}{2} \int_{\epsilon + u}^T (\rho(0^+)A(0^+) - \rho(0^-)A(0^-))
G(s-u,0,y) \frac{\partial \varphi(s,0)}{\partial x} ds \\
&\rightarrow \int_\R \delta(x,y)\, \varphi(u,x)\,\rho(x)\,dx \\
&\quad   + \frac{1}{2} \int_{u}^T (\rho(0^+)A(0^+) - \rho(0^-)A(0^-))
G(s-u,0,y) \frac{\partial \varphi(s,0)}{\partial x} ds.
\end{align*}
Consequently,
 \begin{multline*}
\int_{(0,T)\times \R} u(s,x)\,\mathcal{H}\varphi(s,x)\,\rho(x)\,dx\,ds = \int_{(0,T)\times \R}\varphi(s,y)\,\rho(y)\,W(ds\,dy) \\
 + \int_{(0,T)\times \R} \frac{1}{2} \int_{u}^T (\rho(0^+)A(0^+) - \rho(0^-)A(0^-))
G(s-u,0,y) \frac{\partial \varphi(s,0)}{\partial x} ds W(du, dy).
\end{multline*}
Applying the Fubini  theorem \citep[Th.~2.6]{W} with $ G = [0,T]$,  $ \d\nu = ds$, $ M = W$ and
\[
g(u,s,y)=
{\mathbf 1}_{[0,s]}(u) G(s-u,0,y) \frac{\partial \varphi (s,0)}{\partial x}
\]
we get
\begin{align*}
\MoveEqLeft \int_{(0,T)\times \R^\star } u(s,x)\,\mathcal{H}\varphi(s,x)\,\rho(x)\,dx\,ds = \int_{(0,T)\times \R}\varphi(s,y)\,\rho(y)\,W(ds\,dy) \\
&\quad  + \frac{1}{2}   \int_{0}^T (\rho(0^+)A(0^+) - \rho(0^-)A(0^-)) \int_{(0,s)\times \R}
G(s-u,0,y)  W(du, dy) \frac{\partial \varphi(s,0)}{\partial x} ds\\
&=  \int_{(0,T)\times \R}\varphi(s,y)\,\rho(y)\,W(ds\,dy) + \frac{1}{2}   \int_{0}^T (\rho(0^+)A(0^+) - \rho(0^-)A(0^-)) u(s,0) \frac{\partial \varphi(s,0)}{\partial x} ds.
\end{align*}
\end{proof}

\subsection{Weak solution implies mild solution}

In this section we prove that every weak solution is a mild solution.
In view of Theorem~\ref{th:m->w}, it suffices to prove the uniqueness of a weak solution.



%
\begin{theorem}
The mild solution is a unique weak solution to the stochastic partial differential equation \eqref{e1:1}.
\end{theorem}

\begin{proof}
Assume that $u_1$ and $u_2$ are two weak solutions to \eqref{e1:1}.
Fix some $T > 0$.
Then \eqref{weaksol} implies that for every $\varphi\in C^{\infty}_c([0,T]\times\R)$,
\begin{multline}\label{weaksol-diff}
-\int_{[0,T] \times \R ^\star }\bigl(u_1(s,x) - u_2(s,x)\bigr)
\left [ \frac{\partial \varphi(s, x)}{\partial s}+ \mathcal{L}\varphi(s,x)\right ]\,\rho (x)\,dx\,ds\\
= \frac{1}{2} \int_{0}^T \bigl(\rho (0^+) A(0^+) -\rho (0^-) A(0^-)\bigr)\bigl(u_1(s,0) - u_2(s,0)\bigr)\frac{\partial \varphi(s,0)}{\partial x} \,ds \quad\text{a.\,s.}
\end{multline}

First, let us prove that $u_1=u_2$ in $(0,T)\times(0,+\infty)$.
Let $\psi$ be an arbitrary function from $C^\infty_c((0,T)\times(0,+\infty))$.
We may define it on the entire $[0,T]\times\R$ by putting $\psi=0$ outside $(0,T)\times(0,+\infty)$.
We claim that there exists
$\varphi \in C^{1,2}([0,T]\times\R)$
such that
\begin{equation}\label{eq:pde-bw}
\frac{\partial\varphi(t,x)}{\partial t} + \frac12 a_2\frac{\partial^2\varphi(t,x)}{\partial x^2} = \psi(t,x),
\quad t\in[0,T],\; x\in\R.
\end{equation}
Indeed, it is well known \citep[see, e.\,g.][Th.~2, p.~50]{Evans} that for any
$f\in C^{1,2}(\R_+\times\R)$ with compact support
the function
\[
v(t,x) = \int_0^t\frac{1}{2\sqrt{\pi(t-s)}}\int_{\R} e^{-\frac{|x-y|^2}{4(t-s)}} f(s,y)\,dy\,ds
\]
belongs to $C^{1,2}((0,+\infty)\times\R)$ and satisfies the heat equation
\begin{equation}\label{eq:pde-heat}
\frac{\partial v}{\partial t} - \frac{\partial^2 v}{\partial x^2} = f,
\quad t>0,\: x\in\R.
\end{equation}
Then the desired solution to \eqref{eq:pde-bw} can be constructed from the solution to \eqref{eq:pde-heat} by the time change $s=\frac12a_2(T-t)$.

Assume for definiteness that $\supp \psi\subset[\sigma,\tau]\times[a,b]\subset(0,T)\times(0,+\infty)$.
Using regularization (see \citet[Sec.~1.3]{Hormander} or \citet[Sec.~1.2]{Vlad}), we can approximate the function
$\tilde\varphi=\varphi \ind_{[\sigma,\tau]\times[a,b]}$ by functions
$\varphi_\varepsilon\in C^\infty_c([0,T]\times\R)$
in such a way that, for sufficiently small $\varepsilon$, $\supp \varphi_\varepsilon \subset (0,T)\times(0,+\infty)$,
\[
\norm{\varphi_\varepsilon}_{L_\infty} \le \norm{\tilde\varphi}_{L_\infty},
\quad
\norm{\frac{\partial\varphi_\varepsilon}{\partial t}}_{L_\infty} \le \norm{\frac{\partial\tilde\varphi}{\partial t}}_{L_\infty},
\quad
\norm{\frac{\partial^2\varphi_\varepsilon}{\partial x^2}}_{L_\infty} \le \norm{\frac{\partial^2\tilde\varphi}{\partial x^2}}_{L_\infty},
\]
and the convergences
\[
\varphi_\varepsilon \to \varphi\ind_{[\sigma,\tau]\times[a,b]},
\quad
\frac{\partial\varphi_\varepsilon}{\partial t} \to \frac{\partial\varphi}{\partial t}\ind_{[\sigma,\tau]\times[a,b]},
\quad
\frac{\partial^2\varphi_\varepsilon}{\partial x^2} \to \frac{\partial^2\varphi}{\partial x^2}\ind_{[\sigma,\tau]\times[a,b]},
\quad \text{as }\varepsilon\downarrow0,
\]
hold almost everywhere in $[0,T]\times\R$.

By inserting this $\varphi_\varepsilon$ into \eqref{weaksol-diff} we get
\[
\int_{(0,T) \times (0,+\infty) }\bigl(u_1(s,x) - u_2(s,x)\bigr)\left [\frac{\partial\varphi_\varepsilon(t,x)}{\partial t} + \frac12 a_2\frac{\partial^2\varphi_\varepsilon(t,x)}{\partial x^2}\right ]\,dx\,ds
= 0 \quad\text{a.\,s.}
\]
Note that the expression in square brackets is bounded and compactly supported, and $u_1$, $u_2$ are locally integrable.
Hence, letting $\varepsilon\downarrow0$, we obtain by the Lebesgue dominated convergence theorem that
\[
\int_{(0,T) \times (0,+\infty) }\bigl(u_1(s,x) - u_2(s,x)\bigr) \psi(s, x)\,dx\,ds
= 0 \quad\text{a.\,s.}
\]
Since $\psi\in C^\infty_c((0,T)\times(0,+\infty))$ is arbitrary and $u_1$, $u_2$ are locally integrable, we have that
$u_1-u_2=0$ almost everywhere in $(0,T) \times (0,+\infty)$ a.\,s.  \citep[see, e.\,g.,][Th.~1.2.5]{Hormander}.

Similarly, one can prove that  $u_1=u_2$ almost everywhere in $(0,T)\times(-\infty,0)$ a.\,s., whence the result follows.
\end{proof}

\section*{Acknowledgments}
YM and KR acknowledge that the present research is carried through within the frame and support of the ToppForsk project nr. 274410 of the Research Council of Norway with title STORM: Stochastics for Time-Space Risk Models.

\end{document}